\documentclass[10pt,a4paper]{article}

\usepackage[a4paper,hmargin=38mm,vmargin=40mm,footskip=15mm]{geometry}
\usepackage[OT1]{fontenc}
\usepackage[utf8]{inputenc}
\usepackage[english]{babel}
\usepackage{csquotes}
\usepackage{lmodern}
\usepackage{amsmath}
\usepackage{amsfonts}
\usepackage{amssymb}
\usepackage{mathtools}
\usepackage{amsthm}
\usepackage{thmtools}
\usepackage{bbm}
\allowdisplaybreaks
\usepackage{float}
\usepackage[font={small,footnotesize}]{caption}
\usepackage{subcaption}
\usepackage{booktabs}
\usepackage{tabularx}
\usepackage{tikz}
\usepackage{pgfplots}
\usepackage{color}
\pgfplotsset{compat=1.14}
\usetikzlibrary{arrows.meta}
\usetikzlibrary{patterns}
\usepackage{enumitem}
\usepackage[backend=biber,style=numeric,sorting=nyt,sortcites=true,doi=false,isbn=false,url=false,eprint=true,maxbibnames=9,giveninits=true]{biblatex}
\DeclareNameAlias{sortname}{last-first}
\renewbibmacro{in:}{}
\AtEveryBibitem{%
  \ifentrytype{misc}{%
  }{%
    \clearfield{archivePrefix}%
    \clearfield{arxivId}%
    \clearfield{eprint}%
    \clearfield{primaryClass}%
  }%
}
\usepackage{hyperref}
\hypersetup{
colorlinks=false,
bookmarksnumbered=false,
pdfpagemode=UseNone,
pdfstartpage=1,
pdfstartview={XYZ null null 1.00},
pdfpagelayout=OneColumn,
pdflang=English}
\usepackage{authblk}
\usepackage{todonotes}

\theoremstyle{plain}
\newtheorem{theorem}{Theorem}

\theoremstyle{definition}

\newtheorem{examplex}{Example}
{\pushQED{\qed}\examplex}
{\popQED\endexamplex\vspace{2ex}}

\setlist{align=right,labelindent=0.2em,labelwidth=1.5em,labelsep*=0.6em,leftmargin=!,topsep=0ex,partopsep=0ex,parsep=0ex,itemsep=0ex}
\setlist[itemize,1]{label=\raisebox{0.0em}{\scalebox{1.0}{$\bullet$}}}
\setlist[itemize,2]{label=\raisebox{0.1em}{\scalebox{0.5}{$\blacksquare$}}}
\setlist[itemize,3]{label=\raisebox{0.1em}{\scalebox{0.7}{$\blacktriangleright$}}}
\setlist[itemize,4]{label=\raisebox{0.1em}{\scalebox{0.6}{$\blacklozenge$}}}
\setlist[enumerate]{label=(\alph*)}

\newcommand{\citelink}[2]{\nocite{#1}\hyperlink{cite.\therefsection @#1}{#2}}
\renewcommand{\P}{\mathord{\mathbb{P}}}
\newcommand{\E}{\mathord{\mathbb{E}}}

\newcommand{\N}{\mathord{\mathbb{N}}}
\newcommand{\e}{\mathalpha{\mathrm{e}}}

\newcommand{\smallO}{\mathord{o}}

\newcommand{\rg}{\mathord{\mathbb{G}}}
\newcommand{\W}{W}

\newcommand{\pmax}{p_{\textup{max}}}

\newcommand{\cn}{\omega}

\newcommand{\qcn}[1][\gamma]{\omega^{\mspace{-1mu}\scriptscriptstyle(\mspace{-2.5mu}#1\mspace{-2mu})\mspace{-1mu}}}
\newcommand{\tqcn}[1][n]{\qcn_{#1}}
\newcommand{\kernel}{\kappa}
\renewcommand{\epsilon}{\varepsilon}
\renewcommand{\phi}{\varphi}

\title{Quasi-cliques in inhomogeneous random graphs}
\author{Kay Bogerd\\\texttt{\small\href{mailto:k.m.bogerd@tue.nl}{k.m.bogerd@tue.nl}}}
\affil{Eindhoven University of Technology}
\date{\today}

\begin{document}
\noindent\makebox[\textwidth][c]{\begin{minipage}[c]{1.1\textwidth}\maketitle\end{minipage}}

\begin{abstract}
Given a graph $G$ and a constant $\gamma \in [0,1]$, let $\qcn(G)$ be the largest integer $r$ such that there exists an $r$-vertex subgraph of $G$ containing at least $\gamma \smash{\binom{r}{2}}$ edges. It was recently shown that $\qcn(G)$ is highly concentrated when $G$ is an Erd\H{o}s-R\'{e}nyi random graph \citelink{Balister2018}{(Balister, Bollobás, Sahasrabudhe, Veremyev, 2019)}.
This paper provides a simple method to extend that result to a setting of inhomogeneous random graphs, showing that $\qcn(G)$ remains concentrated on a small range of values even if $G$ is an inhomogeneous random graph. Furthermore, we give an explicit expression for $\qcn(G)$ and show that it depends primarily on the largest edge probability of the graph $G$.
\end{abstract}

\section{Introduction}
\label{sec:introduction}
Let $G = (V, E)$ be a simple graph, with vertex set $V$ and edge set $E$. Given a subset of vertices $S \subseteq V$, let $G[S]$ denote the \emph{subgraph} of $G$ induced by $S$. That is, $G[S]$ is a graph with vertex set $S$ and edge set $\{(i, j) : i, j \in S\} \cap E$. A \emph{clique} is a subset of vertices $C \subseteq V$ such that $G[C]$ is a complete graph, meaning that all vertices in $G[C]$ are connected by an edge. Cliques are an important concept in graph theory, and are often used as a model for community structure \cite{Luce1950,Alba1973,Mokken1979}. In particular, the problem of finding the largest clique or largest community in a given graph has received much interest \cite{Dekel2014,Deshpande2015}.

However, for many practical applications the definition of a clique can be too restrictive. Often a few missing edges within a community are fine, as long as the community remains sufficiently well connected.
To this end, several relaxations have been proposed for the definition of a clique \cite{Pattillo2013a}. One of the most successful of these is known as the \emph{$\gamma$-quasi-clique}, where $\gamma$ is a parameter \cite{Abello1999}. For $\gamma \in [0, 1]$, a $\gamma$-quasi-clique is a subset of vertices $S \subseteq V$ such that $G[S]$ contains at least $\smash{\gamma \binom{|S|}{2}}$ edges. That is, a $\gamma$-quasi-clique is a subset of vertices such that a fraction $\gamma$ of all possible edges between them is present.

Just as for cliques, one would like to know the size of the largest quasi-clique in a given graph \cite{Abello2002,Brunato2008,Veremyev2016}. However, it comes as no surprise that finding the largest quasi-clique is a computationally hard problem \cite{Pattillo2013,Pastukhov2018}, similar to the problem of finding the largest clique \cite{Karp1972,Feige1991,Hastad1999}. To circumvent this difficulty, a common approach has been to study the related problem of determining the size of the largest clique or quasi-clique in random graphs. For cliques this approach has been very fruitful, and it turns out that the size of the largest clique is highly concentrated in a variety of random graph models. The first results of this type were obtained for Erd\H{o}s-R\'{e}nyi random graphs \cite{Matula1972,Matula1976,Bollobas1976,McDiarmid1984}, and later similar results were obtained for random geometric graphs \cite{Muller2008}, and inhomogeneous random graphs \cite{Dolezal2017,Bogerd2018}.

Recently, the size of the largest quasi-clique was also studied in an Erd\H{o}s-R\'{e}nyi random graph, where it was shown that the largest quasi-clique is again highly concentrated \cite{Balister2018}. The aim of this paper is to extend that result to the setting of inhomogeneous random graphs.
In particular, we formalize a heuristic presented in \cite{Bogerd2018}, and show how this (together with the result from \cite{Balister2018}) can be applied to show that the largest quasi-clique remains concentrated on a narrow range of values even in an inhomogeneous random graph.

\section{Model and results}
\label{sec:mode4l_and_results}
We are interested in understanding the behavior of the largest quasi-clique in an inhomogeneous random graph. To this end, define the $\gamma$-quasi-clique number $\qcn(G)$ of a graph $G$ as the size of the largest subset of vertices $S \subseteq V$ such that the induced subgraph $G[S]$ contains at least $\smash{\gamma \binom{|S|}{2}}$ edges, where $\gamma \in [0, 1]$ is a parameter. Note that $\qcn[1](G)$ is the familiar clique number of $G$, usually denoted simply by $\cn(G)$.

In this paper, we study the behavior of $\qcn(G)$ when $G$ is distributed according to the random graph model $\rg(n, \kernel)$. This model has two parameters: the number of vertices $n$, and a symmetric measurable function called a \emph{kernel} $\kernel : [0, 1]^2 \to (0, 1)$. Below we introduce the key concepts of this model, for a more detailed overview we refer the reader to Lov\'{a}sz's book \cite{Lovasz2012}. An element of $\rg(n, \kernel)$ is a simple graph $G = (V, E)$ that has $n \in \N$ vertices with vertex set $V = [n] \coloneqq \{1, \ldots, n\}$, and a random edge set $E$. Each vertex $i \in V$ is assigned a \emph{weight} $\W_i$, which is simply a uniform variable on $[0, 1]$, that is $\W_i \sim \text{Unif}(0,1)$. Conditionally on these weights, the presence of an edge between two vertices $i,j \in V$, with $i \neq j$, is modeled by independent Bernoulli random variables with success probability
\begin{equation}
\label{eq:edge_probabilities}
p_{ij} \coloneqq \P\bigl((i,j) \in E \,|\, (\W_k)_{k \in V}\bigr) = \kernel(\W_i, \W_j) \,.
\end{equation}

The kernel $\kernel(\cdot, \cdot)$ and the vertex weights $\W_i$ are both not allowed to depend on the graph size $n$, and therefore the edge probabilities $p_{ij}$ are independent of $n$. This means that the graphs we consider are necessarily dense and have a number of edges that is quadratic in the graph size.

This brings us to the main result of this paper, which is to show that the $\gamma$-quasi-clique number $\qcn(G)$ of a graph $G \sim \rg(n, \kernel)$ is concentrated on a small range of values. Furthermore, this result shows that the size of the largest quasi-clique depends primarily on the densest part of the graph, where the edge probabilities are close to their maximum value. This is made precise by the following result.
\begin{theorem}
\label{thm:quasi_clique_number_concentration}
Let $\kernel(\cdot, \cdot)$ be a kernel that is continuous and attains it maximum value at the point $(c, c)$ for some $c \in [0, 1]$, and let $\pmax \coloneqq \kernel(c, c)$. Given $\pmax < \gamma \leq 1$, define
\begin{equation}
\label{eq:typical_quasi_clique_number}
\tqcn
  \coloneqq \frac{2 \log(n)}{D(\gamma, p_\textup{max})} \,,
\end{equation}
where $D(\gamma, p)$ is the Kullback-Leibler divergence between the Bernoulli distributions $\text{Bern}(\gamma)$ and $\text{Bern}(p)$, given by
\begin{equation}
D(\gamma, p)
  \coloneqq \begin{cases}
  \gamma \log\mspace{-2mu}\left(\frac{\gamma}{p}\right) + (1 - \gamma) \log\mspace{-2mu}\left(\frac{1 - \gamma}{1 - p}\right) & \;\;\text{if}\;\; \gamma < 1 \,,\\
  \log\mspace{-2mu}\left(\frac{1}{p}\right) & \;\;\text{if}\;\; \gamma = 1 \,.
  \end{cases}
\end{equation}
Then, for every $\epsilon > 0$,
\begin{equation}
\label{eq:quasi_clique_number_concentration}
\P\Bigl(\qcn(G) \in \bigl[(1 - \epsilon)\tqcn, (1 + \epsilon)\tqcn\bigr]\Bigr) \to 1 \,,
\qquad\text{as } n \to \infty \,.
\end{equation}
\end{theorem}

To display the applicability of the above result, we show that it can be applied to many well-known random graph models. The simplest example is probably the Erd\H{o}s-R\'{e}nyi random graph, which is obtained by setting the kernel $\kernel(x, y)$ to a constant independent of $x$ and $y$.
Another commonly used example are the so-called rank-1 random graphs, where $\kernel(x, y) = \phi(x) \mspace{1mu} \phi(y)$ for some function $\phi$. Often the function $\phi(\cdot) = F_X^{-1}\mspace{-1mu}(\cdot)$ is the inverse cumulative distribution function of some distribution $X$, so that $\phi(\W_i)$ can be interpreted as a sample from that distribution. This results in a model similar to that considered in \cite{Bogerd2018}. The final model that satisfies the conditions in Theorem~\ref{thm:quasi_clique_number_concentration} is the stochastic block model \cite{Holland1983}, also called the planted partition model in computer science. This model is obtained when the kernel $\kernel(\cdot, \cdot)$ is only allowed to take on finitely many different values.

Note that Theorem~\ref{thm:quasi_clique_number_concentration} gives the first-order behavior of $\tqcn$ from \eqref{eq:typical_quasi_clique_number}.
More precise results are known for the clique and quasi-clique number in an Erd\H{o}s-R\'{e}nyi random graph \cite{Balister2018,Matula1976}, or for the clique number in rank-1 random graphs \cite{Bogerd2018}. Specifically, in those cases the quasi-clique number and clique number are concentrated on two consecutive integers. Therefore, it might be reasonable to expect that it is likewise possible to show such a two-point concentration result in the more general model we consider in this paper. However, this would require a significantly more detailed analyses. The main difficulty here is that the higher order terms of $\tqcn$ will likely depend in a complex way on the whole kernel $\kernel(\cdot, \cdot)$ and not just on the maximum value $\kernel(c, c)$. This was also observed for rank-1 random graphs in \cite{Bogerd2018}, where several examples are explicitly computed. Thus, the method we use in the proof of Theorem~\ref{thm:quasi_clique_number_concentration} will likely not be precise enough to characterize the higher order terms of $\tqcn$ and a different approach would be needed for this.

We end this paper with the proof of Theorem~\ref{thm:quasi_clique_number_concentration}. This proof is based on the ideas presented in \cite[Section~3.1]{Bogerd2018} combined with the results in \cite{Balister2018} and \cite{Matula1976}.

\mathtoolsset{showonlyrefs}
\begin{proof}[Proof of Theorem~\ref{thm:quasi_clique_number_concentration}]
Below we consider the upper and lower bound of \eqref{eq:quasi_clique_number_concentration} separately. Furthermore, we will use the following standard asymptotic notation: given deterministic sequences $a_n$ and $b_n$, we write $a_n = \smallO(b_n)$ when $a_n / b_n \to 0$, and we say that a sequence of events holds with high probability if it holds with probability tending to $1$. When limits are unspecified they are taken as the number of vertices $n$ tends to $\infty$.

\vspace{-6pt}
\paragraph{\normalfont\emph{Upper bound:}}
We first define a coupling between the random graph $\rg(n; \kernel)$ and the Erd\H{o}s-R\'{e}nyi random graph $\rg(n; \pmax)$, where we recall that $\pmax = \kernel(c, c)$ is the maximum edge probability. For $i \neq j \in [n]$, let $U_{ij} \sim \text{Unif}(0, 1)$ be independent uniform random variables on $[0, 1]$. Conditionally on these uniform random variables and the weights $\W_i$, with $i \in [n]$, define
\begin{alignat}{7}
\label{eq:upper_bound_coupling}
G &= (V, E) \,,
\quad&&\text{with}&\quad V &= [n]\,,
\quad&&\text{and}&\quad  E &= \{(i, j) : U_{ij} \leq \kernel(\W_i, \W_j)\} \,,\\
G' &= (V', E') \,,
\quad&&\text{with}&\quad V' &= [n]\,,
\quad&&\text{and}&\quad  E' &= \{(i, j) : U_{ij} \leq \kernel(c, c)\} \,.
\end{alignat}
It can easily be seen that $G$ is an inhomogeneous random graph, that is $G \sim \rg(n, \kernel)$. Similarly, $G' \sim \rg(n, \pmax)$ is distributed as an Erd\H{o}s-R\'{e}nyi random graph with edge probability $\pmax = \kernel(c, c)$.

Because the edge probabilities satisfy $p_{ij} = \kernel(\W_i, \W_j) \leq \pmax$ almost surely, for all $i \neq j \in [n]$, the coupling in \eqref{eq:upper_bound_coupling} shows that $\qcn(G) \leq \qcn(G')$ almost surely. Furthermore, by \cite[Theorem~1]{Balister2018} if $\gamma < 1$ or \cite[Theorem~6]{Matula1976} if $\gamma = 1$, it follows that
\begin{equation}
\qcn(G')
  \leq \frac{2}{D(\gamma, \pmax)} \Bigl(\log(n) - \log\log(n) + \log(\e \mspace{1mu} D(\gamma, \pmax) / 2)\Bigr) + 1 + \epsilon \,,
\end{equation}
with high probability.

Combining the above, we obtain
\begin{align}
\qcn(G)
  &\leq \qcn(G')\\
  &\leq \frac{2}{D(\gamma, \pmax)} \Bigl(\log(n) - \log\log(n) + \log(\e \mspace{1mu} D(\gamma, \pmax) / 2)\Bigr) + 1 + \epsilon\\
  &\leq (1 + \epsilon)\frac{2 \log(n)}{D(\gamma, \pmax)}
  = (1 + \epsilon)\tqcn \,,
\end{align}
with high probability. This shows that $\P\left(\qcn(G) \leq (1 + \epsilon)\tqcn\right) \to 1$, completing the proof for the upper bound of \eqref{eq:quasi_clique_number_concentration}.

\vspace{-6pt}
\paragraph{\normalfont\emph{Lower bound:}}
Let $\delta_n = 1 / \log(n)$ and define $S_n \coloneqq \{i \in V : \W_i \in [c - \delta_n, c + \delta_n]\}$ to be the subset of vertices that have vertex weight $\W_i$ close to $c$, where we recall that $c$ is such that the kernel $\kernel(\cdot, \cdot)$ attains it maximal value at the point $(c, c)$. Note that the set $S_n$ is random and by Hoeffding's inequality (see \cite[Theorem~2.8]{Boucheron2013}), for any $t > 0$, we have 
\begin{equation}
\label{eq:lower_bound_hoeffding}
\P\left(|S_n| \geq \E[|S_n|] - t\right) \leq \exp\left(-2 t^2 / n\right) \to 0 \,,
\end{equation}
where $\E[|S_n|] = n \mspace{1mu} \P(W \in [c - \delta_n, c + \delta_n]) = n^{1 - \smallO(1)}$ by definition of $\delta_n$. Furthermore, define
$p_n \coloneqq {\textstyle\inf_{\cramped{(x,y) \in [c - \delta_n, c + \delta_n]^2 \cap [0, 1]^2}}} \, \kernel(x, y)$ and observe that $p_n \to \pmax$ by continuity of the kernel, and thus $D(\gamma, p_n) \to D(\gamma, \pmax)$. Using this, together with \eqref{eq:lower_bound_hoeffding} and $t$ fixed, we obtain
\begin{align}
(1 - \epsilon)\frac{2 \log(n)}{D(\gamma, \pmax)}
  &\leq (1 - \epsilon/2)\frac{2 \log(\E[|S_n|] - t)}{D(\gamma, \pmax)}\\
  &\leq (1 - \epsilon/3)\frac{2 \log(|S_n|)}{D(\gamma, \pmax)}\\
\label{eq:lower_bound_hoeffding_tqcn}
  &\leq (1 - \epsilon/4)\frac{2 \log(|S_n|)}{D(\gamma, p_n)} \,,
\end{align}
with high probability.

Similarly to the coupling in \eqref{eq:upper_bound_coupling}, conditionally on the uniform random variables $U_{ij}$, for $i \neq j \in [n]$, and the vertex weights $\W_i$, for $i \in [n]$, define
\begin{alignat}{7}
\label{eq:lower_bound_coupling}
G'' &= (V'', E'') \,,
\quad&&\text{with}&\quad V'' &= [n]\,,
\quad&&\text{and}&\quad  E'' &= \{(i, j) : U_{ij} \leq p_n\} \,.
\end{alignat}
Note that the graph $G''$ is distributed as the Erd\H{o}s-R\'{e}nyi random graph $\rg(n, p_n)$ with edge probability $p_n$.

Given a graph $G$, recall that $G[S_n]$ denotes the subgraph induced by the vertices in $S_n$. Because the kernel is continuous around the point $(c, c)$, there exists an $n$ large enough such that $\delta_n$ is small enough to ensure that the edge probabilities satisfy $p_{ij} \geq p_n$ almost surely, for all $i \neq j \in S_n$ (note that, if the kernel is continuous everywhere then this holds for every $n$). Hence, the coupling in \eqref{eq:lower_bound_coupling} shows that $\qcn(G) \geq \qcn(G[S_n]) \geq \qcn(G''[S_n])$ almost surely, provided $n$ is large enough. Combining this with \eqref{eq:lower_bound_hoeffding_tqcn} and \cite[Theorem~1]{Balister2018} if $\gamma < 1$ or \cite[Theorem~6]{Matula1976} if $\gamma = 1$, we obtain
\begin{align}
\qcn(G)
  &\geq \qcn(G[S_n])
  \geq \qcn(G''[S_n])\\
  &\geq \frac{2}{D(\gamma, p_n)} \Bigl(\log(|S_n|) - \log\log(|S_n|) + \log(\e D(\gamma, p_n) / 2)\Bigr) - \epsilon\\
  &\geq (1 - \epsilon/4)\frac{2 \log(|S_n|)}{D(\gamma, p_n)}
  \geq (1 - \epsilon)\frac{2 \log(n)}{D(\gamma, \pmax)}
  = (1 - \epsilon)\tqcn \,,
\end{align}
with high probability. This shows that $\P\left(\qcn(G) \geq (1 - \epsilon)\tqcn\right) \to 1$, completing the proof for the lower bound of \eqref{eq:quasi_clique_number_concentration}.
\end{proof}

\paragraph{Acknowledgements.} 
The author thanks his supervisors Remco~van~der~Hofstad and Rui~M.~Castro for extensive proofreading and providing valuable feedback.

{\setlength{\emergencystretch}{3em}\printbibliography[title={References}]}
\end{document}